\documentclass{amsart}
\usepackage{amsmath}
\usepackage{enumerate}
\usepackage{amssymb}
\usepackage{xcolor}

\title{Graphs with $\alpha_1$ and $\tau$ Both Large}
\author{Gregory J.~Puleo}
\renewcommand{\epsilon}{\varepsilon}
\renewcommand{\subset}{\subseteq}
\DeclareMathOperator{\bin}{Bin}

\newcommand{\join}{\vee}

\newcommand{\erdos}{Erd\H os}

\newcommand{\bad}[1]{\textcolor{red}{\emph{#1}}}
\newcommand{\citethis}[1]{\bad{[CITE THIS]}}

\newcommand{\sizeof}[1]{\left\lvert{#1}\right\rvert}
\newcommand{\floor}[1]{\left\lfloor{#1}\right\rfloor}

\newcommand{\aph}{\alpha_1}

\let\oldtau\tau
\renewcommand{\tau}{\oldtau_1}

\newcommand{\pr}{\mathbb{P}}
\newcommand{\ex}{\mathbb{E}}

\newtheorem{proposition}{Proposition}[section]

\newtheorem{lemma}[proposition]{Lemma}
\newtheorem{theorem}[proposition]{Theorem}
\newtheorem{corollary}[proposition]{Corollary}
\newtheorem{problem}[proposition]{Problem}
\theoremstyle{definition}
\newtheorem{definition}[proposition]{Definition}
\theoremstyle{remark}

\begin{document}
\maketitle
\begin{abstract}
  Given a graph $G$, let $\tau(G)$ denote the smallest size of a set
  of edges whose deletion makes $G$ triangle-free, and let $\aph(G)$
  denote the largest size of an edge set containing at most one edge
  from each triangle of $G$. \erdos, Gallai, and Tuza introduced
  several problems with the unifying theme that $\aph(G)$ and $\tau(G)$
  cannot both be ``very large''; the most well-known such problem is
  their conjecture that $\aph(G) + \tau(G) \leq \sizeof{V(G)}^2/4$,
  which was proved by Norin and Sun.  We consider three other problems
  within this theme (two introduced by \erdos, Gallai, and Tuza, 
  another by Norin and Sun), all of which request an upper bound either
  on $\min\{\aph(G), \tau(G)\}$ or on $\aph(G) + k\tau(G)$ for some
  constant $k$, and prove the existence of graphs for which these
  quantities are ``large''.
\end{abstract}
\section{Introduction}
A \emph{triangle independent set} in a graph $G$ is a set of edges
containing at most one edge from each triangle of $G$, while a
\emph{triangle edge cover} in a graph $G$ is a set of edges containing
at least one edge from each triangle of $G$. Equivalently, a triangle
edge cover is a set of edges whose deletion from $G$ results in a
triangle-free graph.  We write $\aph(G)$ to denote the size of a
largest triangle independent set in $G$ and we write $\tau(G)$ for the
size of a smallest triangle edge cover in $G$. \erdos, Gallai, and
Tuza~\cite{EGT} considered several problems relating the quantities
$\aph(G)$ and $\tau(G)$, with the unifying theme that $\aph(G)$ and
$\tau(G)$ should not both be ``large'': informally, if it is easy for
an edge set to avoid all triangles, in the sense of $\aph(G)$ being
large, then it should also be easy to destroy all triangles, so that
$\tau(G)$ should be small.  In particular, they posed the following
statement as a conjecture, which was was proved (in a somewhat
stronger form) by Norin and Sun~\cite{norin-sun}.
\begin{theorem}[Norin--Sun~\cite{norin-sun}]\label{thm:norin-sun}
  If $G$ is an $n$-vertex graph, then $\aph(G) + \tau(G) \leq n^2/4$.
\end{theorem}
The complete graph $K_n$ and the complete bipartite graph
$K_{n/2,n/2}$ both satisfy $\aph(G) + \tau(G) = n^2/4$, but a
different part of the sum dominates for each graph.  While this
conjecture has now been proved, many interesting problems relating
$\aph(G)$ and $\tau(G)$ remain open.

Norin and Sun~\cite{norin-sun} posed the following problem:
\begin{problem}[Question~8 of \cite{norin-sun}]\label{prob:norin}
  Determine the largest constant $c$ such that $\aph(G) + c\tau(G)
  \leq \sizeof{E(G)}$ for every graph $G$.
\end{problem}
Erd\"os, Gallai, and Tuza~\cite{EGT} proved that
$\aph(G) + \tau(G) \leq \sizeof{E(G)}$ for every graph $G$, which
gives a lower bound of $c \geq 1$ in Problem~\ref{prob:norin}.  As
Norin and Sun~\cite{norin-sun} observed, a conjecture of Tuza
mentioned in \cite{OldNew} is equivalent to the claim that we can take
$c \geq 5/3$ in Problem~\ref{prob:norin}.  In Section~\ref{sec:norin},
we prove that $c=1$ is the correct answer to Problem~\ref{prob:norin},
which refutes the conjecture of Tuza.

Erd\H os, Gallai, and Tuza also posed the following
closely related problems. A \emph{triangular} graph is a graph
such that every edge lies in some triangle.
\begin{problem}[Problem~13 of \cite{EGT}]\label{prob:minaphtau}
  Determine the largest constant $c$ for which there exists a
  triangular graph $G$ such that $\min(\aph(G), \tau(G)) \geq
  c\sizeof{E(G)}$.
\end{problem}
\begin{problem}[\cite{EGT}]\label{prob:sumaphtau}
  Determine the largest constant $c'$ for which there exists a
  triangular graph $G$ such that $\aph(G) + 2\tau(G) \geq
  c'\sizeof{E(G)}$.
\end{problem}
Since $\aph(G) + \tau(G) \leq \sizeof{E(G)}$ and
$\tau(G) \leq \frac{1}{2}\sizeof{E(G)}$ for all $G$, we have upper
bounds of $c \leq 1/2$ in Problem~\ref{prob:minaphtau} and
$c' \leq 3/2$ in Problem~\ref{prob:sumaphtau}. If we ignore the
``triangular'' restriction, then by taking a disjoint union of
appropriately sized $K_{s}$ and $K_{t,t}$ one can easily get
$c \geq 1/3 - \epsilon$ in Problem~\ref{prob:minaphtau}; likewise,
taking any triangle-free graph yields $c' \geq 1$ in
Problem~\ref{prob:sumaphtau}. 

In Section~\ref{sec:minaphtau} we give a probabilistic construction yielding
triangular graphs with $c \geq \frac{3 - \sqrt{5}}{2} - \epsilon$ in
Problem~\ref{prob:minaphtau} (where
$\frac{3 - \sqrt{5}}{2} \approx 0.38$) and $c' \geq 3 - \sqrt{3} - \epsilon$ in
Problem~\ref{prob:sumaphtau} (where $3 - \sqrt{3} \approx 1.26$).

\section{Bounding $\aph(G) + c\tau(G)$}\label{sec:norin}
In this section, we settle Problem~\ref{prob:norin} by proving the following
theorem.
\begin{theorem}\label{thm:norin-ctx}
  If $c > 1$, then there is a graph $G$ for which $\aph(G) + c\tau(G) > \sizeof{E(G)}$.
\end{theorem}
Our proof uses the following lemma, which is also used in Section~\ref{sec:minaphtau}.
\begin{definition}
  Let $G$ be a graph and let $k$ be a positive real number.  For each
  $S \subset V(G)$, define
  $\phi_k(S) = k\sizeof{S} - \sizeof{E(G[S])}$.
\end{definition}
\begin{definition}
  The \emph{join} of two graphs $G$ and $H$, written $G \join H$,
  is the graph obtained from their disjoint union by adding all possible
  edges between the vertices of $G$ and the vertices of $H$.
\end{definition}
\begin{lemma}[\cite{greg-vizing}]\label{lem:tritau}
  If $G$ is a triangle-free graph on $n$ vertices and $k$ is a positive integer, then $\tau(\overline{K_k} \join G) =  nk - \max_{S \subset V(G)}\phi_k(S)$.  
\end{lemma}
The function $\phi_k$ was studied by Favaron~\cite{favaron} in connection with a problem
of Fink and Jacobson~\cite{fink-jacobson1,fink-jacobson2} concerning $k$-dependence and $k$-domination.
The notation $\phi_k$ is borrowed from the survey paper \cite{DomSurvey}.
Observe that when $k=1$, the quantity $\max_{S \subset V(G)}\phi_1(S)$ is just the independence number of $G$. Note that while previous definitions of the function $\phi_k$
mostly considered integral values of $k$, here we extend it to allow $k$
to be any positive real number.
\begin{proof}[Proof of Theorem~\ref{thm:norin-ctx}]
  Our construction is essentially the same construction used by
  \erdos, Gallai, and Tuza for the lower bound in Theorem~5 of
  \cite{EGT}. Let $n$ be a positive integer to be determined, and let
  $H$ be an $n$-vertex triangle-free graph whose independence number
  $\alpha(H)$ is minimum. By a result of Kim~\cite{kim-ramsey}, we
  have $\alpha(H) \leq 9\sqrt{n\ln n}$. (However, weaker and easier
  bounds on $\alpha(H)$ would also suffice for this proof; we only
  need $\alpha(H) = o(n)$).

  Let $G = K_1 \join H$. As $H$ is a triangle-independent subgraph of
  $G$, we have $\aph(G) \geq \sizeof{E(H)} = \sizeof{E(G)} - n$.  The
  $k=1$ case of Lemma~\ref{lem:tritau} implies that
  $\tau(G) = n - \alpha(H) = n - o(n)$, so we have
  \begin{align*}
    \aph(G) + c\tau(G) &= (\sizeof{E(G)} - n) + (n - o(n)) + (c-1)(n - o(n)) \\
    &= \sizeof{E(G)} + (c-1)n - o(n).    
  \end{align*}
  Since $c > 1$, for sufficiently large $n$ we have $\aph(G) + c\tau(G) > \sizeof{E(G)}$, as desired.
\end{proof}

\section{A Lower Bound on $\min\{\aph(G), \tau(G)\}$}\label{sec:minaphtau}
\begin{lemma}\label{lem:manyedges}
  Let $\epsilon, \theta \in (0,1)$ be fixed constants, let
  $p(n) = n^{-\theta}$, and let $G \sim G(n,p)$.  With high
  probability,
  \[\sizeof{E(G[S])} \geq (1-\epsilon)p\frac{\sizeof{S}^2}{2}\] for all
  $S \subset V(G)$ such that $\sizeof{S} \geq \epsilon n$.
\end{lemma}
\begin{proof}
  Fix $S \subset V(G)$ with $\sizeof{S} \geq \epsilon n$ and let the
  random variable $X$ denote the number of edges in $S$. We have
  $X \sim \bin({\sizeof{S} \choose 2}, p)$. We may assume that $n$ is
  large enough that ${\epsilon n \choose 2} \geq \epsilon^2n^2/3$.  By
  Chernoff's inequality (as formulated in Corollary~2.3 of \cite{RandomGraphs}), 
  \begin{align*}
    \pr[X < (1-\epsilon/2)\ex[X]] &\leq 2\exp\left(-\frac{(\epsilon/2)^2}{3}\ex[X]\right)\\
    &\leq 2\exp\left(-\frac{\epsilon^2}{12}p{\epsilon n \choose 2}\right) \\
    & \leq 2\exp\left(-\frac{\epsilon^4}{36}n^{2-\theta}\right)\\
    &\ll 2^{-n}.    
  \end{align*}
  For sufficiently large $n$, we have $(1-\epsilon/2)\ex[X] \geq (1-\epsilon)p\frac{\sizeof{S}^2}{2}$.
  The desired claim therefore follows by applying the union bound.  
\end{proof}
\begin{lemma}\label{lem:phiwhp}
  Let $d, \epsilon > 0$ be fixed constants with $\epsilon < 1$, let $p(n) = n^{-\theta}$, where $0 < \theta < 1$, and let $G \sim G(n,p)$. Let $k = dnp$. If $d \geq 2\epsilon(1-\epsilon)$, then with high probability, $\max_{S \subset V(G)}\phi_k(S) \leq \frac{k^2}{2(1-\epsilon)p}$.
\end{lemma}
\begin{proof} 
  If $\sizeof{S} < \epsilon n$, then we have
  \[ \phi_k(S) \leq k\sizeof{S} < k\epsilon n = \frac{k^2\epsilon}{dp} \leq \frac{k^2}{2(1-\epsilon)p}, \]
  as desired.
  Thus, it suffices to consider $S$ with $\sizeof{S} \geq \epsilon n$.
  By Lemma~\ref{lem:manyedges}, with high probability we have $\sizeof{E(G[S])}
  \geq (1-\epsilon)p\frac{\sizeof{S}^2}{2}$ for all $S \subset V(G)$
  with $\sizeof{S} \geq \epsilon n$. Thus, for all such $S$ and for $n$ sufficiently large, we have with high probability
  \[ \phi_k(S) \leq k\sizeof{S} - (1-\epsilon)p\frac{\sizeof{S}^2}{2}. \]
  Letting $f(x) = kx - (1-\epsilon)p\frac{x^2}{2}$, we see that $f$ is maximized at $x = \frac{k}{(1-\epsilon)p}$,
  attaining a maximum value of $k^2/(2(1-\epsilon)p)$. The conclusion follows.
\end{proof}
\begin{lemma}\label{lem:trifreephi}
  Let $\epsilon \in (0,1)$ be a fixed constant and let $d$ be a fixed constant
  with $d \geq 2\epsilon(1-\epsilon)$. Let $p(n) = n^{-3/4}$, and let $k = dnp$. For sufficiently large $n$,
  there exists a triangle-free graph $G$ with no isolated vertices such that:
  \begin{itemize}
  \item $\sizeof{E(G)} \leq (1+\epsilon)p\frac{n^2}{2}$,
  \item $\sizeof{E(G)} \geq (1-\epsilon)p\frac{n^2}{2}$, and
  \item $\max_{S \subset V(G)}\phi_k(S) \leq \frac{k^2}{2(1-\epsilon)p}$.
  \end{itemize}
\end{lemma}
\begin{proof}
  Consider a random graph $G_0$ drawn from $G(n,p)$. Note that $\sizeof{E(G_0)}$ is a
  binomial random variable with $\ex[\sizeof{E(G_0)}] = p{n \choose 2}$; in particular,
  $\ex[\sizeof{E(G_0)}] \to \infty$ as $n \to \infty$. Thus, by Chernoff's inequality,
  for any fixed $\gamma > 0$ we have $(1-\gamma)p{n \choose 2} \leq \sizeof{E(G_0)} \leq (1+\gamma)p{n \choose 2}$
  with high probability, and since $n \ll pn^2$, this implies that with high probability,
  \begin{equation}
    \label{eq:g0edges}
    (1-\epsilon)p\frac{n^2}{2} + n \leq \sizeof{E(G_0)} \leq (1+\epsilon)p\frac{n^2}{2},
  \end{equation}
  since we can choose, say, $\gamma = \epsilon/2$ so that $(1-\gamma)p\frac{n^2}{2} > (1-\epsilon)p\frac{n^2}{2} + n$
  for sufficiently large $n$. Furthermore, applying Lemma~\ref{lem:phiwhp} with the constants $d$ and $\epsilon/2$ implies that with high
  probability,
  \begin{equation}
    \label{eq:g0phi}
    \max_{S \subset V(G_0)}\phi_k(S) \leq \frac{k^2}{2(1-\epsilon/2)p}.
  \end{equation}
  Furthermore, as the expected number of triangles in $G_0$ is at most
  $n^{3/4}$, Markov's inequality implies that with high probability,
  $G_0$ has at most $n$ triangles. Similarly, with high probability
  $G_0$ has no isolated vertices. 

  Thus, for sufficiently large $n$, there is a graph $G_0$ with at
  most $n$ triangles and with no isolated vertices for which
  Inequalities~\eqref{eq:g0edges} and~\eqref{eq:g0phi} both hold. Fix
  such a graph $G_0$, and let $X$ be a smallest set of edges such that
  $G-X$ is triangle-free.  Observe that $\sizeof{X} \leq n$, since $G$
  has at most $n$ triangles, and that $G_0-X$ has no isolated
  vertices, since if $v$ is an isolated vertex in $G_0-X$, then as $v$
  is not isolated in $G_0$, there is some edge $vw \in X$, and
  $G_0 - (X-vw)$ is also triangle-free, contradicting the minimality
  of $X$.

  Let $G = G_0-X$.  As we have removed at most $n$ edges from $G_0$,
  clearly
  \[ (1-\epsilon)p\frac{n^2}{2} \leq \sizeof{E(G)} \leq (1+\epsilon)p\frac{n^2}{2}. \]
  Furthermore, for each $S \subset V(G)$, the value of $\phi_k(S)$
  has increased by at most $n$ relative to its value in $G_0$, so that
  \[
    \max_{S \subset V(G)} \phi_k(S) \leq n + \frac{k^2}{2(1-\epsilon/2)p} \leq \frac{k^2}{2(1-\epsilon)p},
    \]
    where the last inequality holds provided that $n$ is sufficiently
    large, as the gap between $\frac{k^2}{2(1-\epsilon/2)p}$ and
    $\frac{k^2}{2(1-\epsilon) p}$ is a constant factor of $k^2/p$,
    where $k^2/p \gg n$. Thus, for sufficiently large $n$, the
    graph $G$ produced in this manner has the desired properties.
\end{proof}

\begin{theorem}\label{thm:aphtaubound}
  Let $d, \epsilon > 0$ be fixed constants with $\epsilon < 1$. If $d \geq \epsilon(1-\frac{\epsilon}{2})$, then
  there is a triangular graph $H$ such that
    $\frac{\tau(H)}{\sizeof{E(H)}} \geq \frac{2(1-\epsilon)^2d - d^2}{(2d+1)(1+\epsilon)(1-\epsilon)}$
    and
    $\frac{\aph(H)}{\sizeof{E(H)}} \geq \frac{1-\epsilon}{(2d+1)(1+\epsilon)}$.    
\end{theorem}
\begin{proof}
  Let $G$ be a graph satisfying the conclusion of
  Lemma~\ref{lem:trifreephi} for the given values of $d$ and
  $\epsilon$, let $n = \sizeof{V(G)}$, let $p=n^{-3/4}$, let $k = npd$, and let
  $H = \overline{K_{\floor{k}}} \join G$.

  Observe that
  \begin{align*}
    \sizeof{E(H)} &= n\floor{k} + \sizeof{E(G)}\\
                  &\leq nk + \sizeof{E(G)} \\
                  &\leq nk + (1+\epsilon)p\frac{n^2}{2} \\
                  &\leq (1+\epsilon)n^2p\frac{2d+1}{2}.
  \end{align*}
  Since $G$ is triangle-free and $\phi_{\floor{k}}(S) \leq \phi_k(S)$ for all $S$, and since $n(k-1) \geq (1-\epsilon)nk$ for all sufficiently large $n$, applying Lemma~\ref{lem:tritau} yields
  \begin{align*}
    \tau(H) &= n\floor{k} - \max_{S \subset V(G)}\phi_{\floor{k}}(S) \\
    &\geq n(k-1) - \max_{S \subset V(G)}\phi_{k}(S) \\
    &\geq (1-\epsilon)nk - \frac{k^2}{2(1-\epsilon)p} \\
    &= \frac{n^2pd[2(1-\epsilon)^2 - d]}{2(1-\epsilon)}.
  \end{align*}  
  Combining this with the upper bound on $\sizeof{E(H)}$ and simplifying, we have
  \[ \frac{\tau(H)}{\sizeof{E(H)}} \geq \frac{d[2(1-\epsilon)^2-d]}{(1+\epsilon)(1-\epsilon)(2d+1)} = 
  \frac{2(1-\epsilon)^2d - d^2}{(2d+1)(1+\epsilon)(1-\epsilon)}.\]

  This establishes the desired lower bound on $\tau(H)$. For the bound on $\aph(H)$, observe
  that $G$ is a triangle-independent subgraph of $H$, so that
  \[ \aph(H) \geq \sizeof{E(G)} \geq (1-\epsilon)p\frac{n^2}{2}. \]
  Therefore, using the upper bound on $\sizeof{E(H)}$, we have
  \[ \frac{\aph(H)}{\sizeof{E(H)}} \geq \frac{1-\epsilon}{(2d+1)(1+\epsilon)}. \]
  Finally, since $G$ has no isolated vertices, it is easy to see that
  $H$ is triangular. 
\end{proof}
For any fixed $d > 0$, the hypothesis of Theorem~\ref{thm:aphtaubound} holds
for all sufficiently small positive $\epsilon$. Taking limits as $\epsilon \to 0$ gives the following corollary.
\begin{corollary}\label{cor:gamma}
  For every $d > 0$, and every $\gamma > 0$, there is a triangular graph $H$ with
  $\frac{\tau(H)}{\sizeof{E(H)}} \geq \frac{2d-d^2}{2d+1} - \gamma$
  and $\frac{\aph(H)}{\sizeof{E(H)}} \geq \frac{1}{2d+1} - \gamma$.
\end{corollary}
Note that when $d \geq 2$, the lower bound on $\tau(H)/\sizeof{E(H)}$ in Corollary~\ref{cor:gamma} is nonpositive.
Thus, Corollary~\ref{cor:gamma} is only useful for $d \in (0,2)$.
Choosing $d$ to maximize $\frac{2d-d^2}{2d+1}$ yields the following partial
answer to Problem~\ref{prob:minaphtau}.
\begin{corollary}
  For all sufficiently small $\gamma > 0$, there is a triangular graph $H$ with $\frac{\tau(H)}{\sizeof{E(H)}} \geq \frac{3 - \sqrt{5}}{2} - \gamma > 0.38$ and $\frac{\aph(H)}{\sizeof{E(H)}} \geq \frac{1}{\sqrt{5}} - \gamma > 0.44$.
\end{corollary}
\begin{proof}
  Take $d = \frac{-1 + \sqrt{5}}{2}$ in Corollary~\ref{cor:gamma}. 
\end{proof}
Similarly, choosing $d$ to maximize $\frac{1 + 2(2d-d^2)}{2d+1}$ yields
the following partial answer to Problem~\ref{prob:sumaphtau}.
\begin{corollary}
  For all sufficiently small $\gamma > 0$, there is a triangular graph $H$ with $\frac{\aph(H) + 2\tau(H)}{\sizeof{E(H)}} \geq 3 - \sqrt{3} - \gamma > 1.26$.
\end{corollary}
\begin{proof}
  Take $d = \frac{-1 + \sqrt{3}}{2}$ in Corollary~\ref{cor:gamma}. 
\end{proof}
\section{Acknowledgments}
We thank the anonymous referees for their careful reading of the paper and
their helpful comments.
\bibliographystyle{amsplain}\bibliography{sumbib}
\end{document}